\documentclass[a4paper,10pt]{article}

\usepackage{amsmath}        
\usepackage{amssymb}
\usepackage{amsthm}
\usepackage[nooneline,hang]{subfigure}
\usepackage[font=normalsize,singlelinecheck=off,labelfont=bf]{caption}
\usepackage{graphicx}
\usepackage{bbold}
\usepackage[usenames]{color}
\usepackage{dsfont}
\usepackage{times}
\usepackage[varg]{txfonts}
\usepackage{floatflt}
\usepackage{enumerate}
\usepackage{mathrsfs}
\usepackage{url}

\newcommand{\R}{\mathds R}
\newcommand{\C}{\mathds C}

\newcommand{\Or}{\mathcal O}
\newcommand{\eps}{\varepsilon}

\newtheoremstyle{plainNoItalics}{}{}{\normalfont}{}{\bfseries}{.}{ }{}

\theoremstyle{plain}
\newtheorem{theorem}{Theorem}[section]
\newtheorem{proposition}[theorem]{Proposition}
\newtheorem{corollary}[theorem]{Corollary}
\newtheorem{lemma}[theorem]{Lemma}

\theoremstyle{plainNoItalics}
\newtheorem{remark}[theorem]{Remark}

\newtheorem*{theorem*}{Theorem}
\newtheorem*{proposition*}{Proposition}
\newtheorem*{lemma*}{Lemma}
\newtheorem*{corollary*}{Corollary}
\newtheorem*{remark*}{Remark}

\newtheorem*{observation*}{Observation}
\newtheorem*{example*}{Example}
\newtheorem*{examples*}{Examples}
\newtheorem*{assumption*}{Assumption}

\theoremstyle{definition}
\newtheorem{definition}[theorem]{Definition}
\newtheorem*{definition*}{'Definition'}

\newtheorem*{definitionu*}{Definition}

\title{Approximation of conformal mappings using conformally equivalent 
triangular lattices}
\author{Ulrike B\"ucking}

\AtEndDocument{\bigskip{%
\noindent
 Ulrike B\"ucking $<$buecking@math.tu-berlin.de$>$
\par\medskip

\noindent
Technische Universit\"at Berlin\\
Stra{\ss}e des 17. Juni 136\\
10623 Berlin
}}

\begin{document}

\maketitle

\begin{abstract}
Two triangle meshes are conformally equivalent if their edge lengths are related 
by scale factors associated to the vertices. Such a pair can be considered as 
preimage and image of a discrete conformal map.
In this article we study the approximation of a
given smooth conformal map $f$ by such discrete conformal maps $f^\eps$ 
defined on triangular lattices. In particular, 
let $T$ be an infinite triangulation of the plane with 
congruent strictly acute triangles. We scale this 
triangular lattice by $\eps>0$ and approximate a compact subset 
of the domain of $f$ with a portion of it. For $\eps$ small enough we prove 
that there exists a conformally equivalent triangle mesh whose scale factors 
are given by $\log|f'|$ on the boundary. Furthermore we show that the 
corresponding discrete conformal (piecewise linear) maps $f^\eps$ converge to 
$f$ uniformly in $C^1$ with error of order $\eps$.
\end{abstract}

\section{Introduction}
Holomorphic functions build the basis and heart of the rich theory of complex 
analysis. Holomorphic functions with nowhere vanishing derivative, also called 
{\em conformal maps}, have the property to preserve angles. Thus they may be 
characterized by the fact that they are infinitesimal scale-rotations.

In the discrete theory, the idea of characterizing conformal maps as local
scale-rotations may be translated into different concepts. Here we consider the 
discretization coming from a metric viewpoint: Infinitesimally, lengths are 
scaled by a factor, i.e.\ by $|f'(z)|$ for a conformal function $f$ on 
$D\subset\C$. More generally, on a smooth manifold two Riemannian metrics $g$ 
and $\tilde{g}$ are conformally equivalent if $\tilde{g}=\text{e}^{2u}g$ for 
some smooth function $u$.

The smooth complex domain (or manifold) is replaced in this discrete 
setting by 
a triangulation of a connected subset of the plane $\C$ (or a triangulated 
piecewise Euclidean manifold). 

\subsection{Convergence for discrete conformal PL-maps on triangular lattices}
In this article we focus on the case where the triangulation is a (part of a) 
{\em triangular lattice}. In particular, let $T$ be a lattice triangulation of 
the whole complex plane $\C$ with congruent triangles, see 
Figure~\ref{FigTiling}. 
\begin{figure}[t]
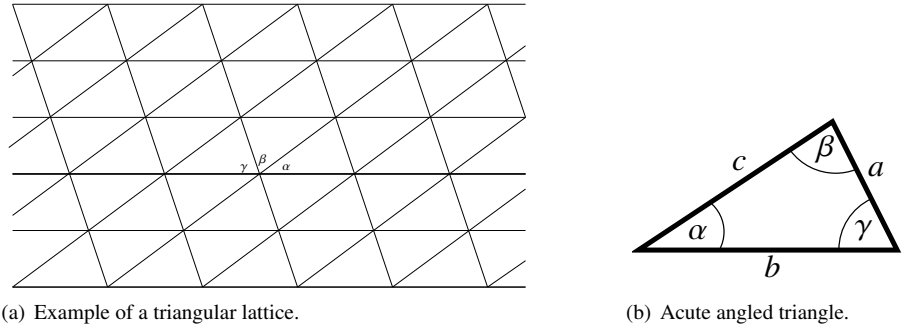

\subfigure[Example of a triangular lattice.]{\label{FigTiling}
\scalebox{0.25}{\input{Triangulation.pspdftex}}
}
\hspace{3em}
\subfigure[Acute angled triangle.]{\label{figTriang}
 \input{TriangleABC.pspdftex}
}
\caption{Lattice triangulation of the plane with congruent 
triangles.}\label{FigRegTile}
\end{figure}
The sets of vertices and edges of $T$ are denoted by $V$ and $E$ respectively.
Edges will often be written as $e=[v_i,v_j]\in E$, where $v_i,v_j\in V$ are 
its incident vertices. For triangular faces we use the notation 
$\Delta[v_i,v_j,v_k]$ enumerating the incident vertices with respect to the
orientation (counterclockwise) of $\C$.

On a subcomplex of $T$ we now define a 
discrete conformal mapping. The main idea is to change the lengths of the edges 
of the triangulation according to scale factors at the vertices. 
The new triangles are then ``glued together to result in a piecewise linear 
map, see Figure~\ref{FigExMap} for an illustration. More precisely, we have

\begin{definition}\label{defdcm}
 A {\em discrete conformal PL-mapping}\index{Discrete conformal mapping} $g$ is 
a continuous and orientation preserving map of a subcomplex $T_S$ of a 
triangular lattice $T$ to $\C$ which is locally a homeomorphism in a 
neighborhood of each interior point and whose restriction to every triangle
is a linear map onto the corresponding image triangle, that is the mapping is 
piecewise linear. Furthermore, there exists a function
$u:V_S\to\R$ on the vertices, 
called {\em associated scale factors}, such that for all edges $e=[v,w]\in E_S$ 
there holds
\begin{equation}\label{eqdefdiscf}
 |g(v)-g(w)|=|v-w|\text{e}^{(u(v)+u(w))/2},
\end{equation}
where $|a|$ denotes the modulus of $a\in\C$.
\end{definition}

Note that equation~\eqref{eqdefdiscf} expresses a linear relation for the 
logarithmic edge lengths, that is
\begin{equation*}\label{eqceqlog}
2\log |g(v)-g(w)| =2\log |v-w| +u(v)+u(w).
\end{equation*}

\begin{figure}[h]
\centering{
\includegraphics[width=0.3\textwidth, trim={0 2cm 0 2.7cm},clip]{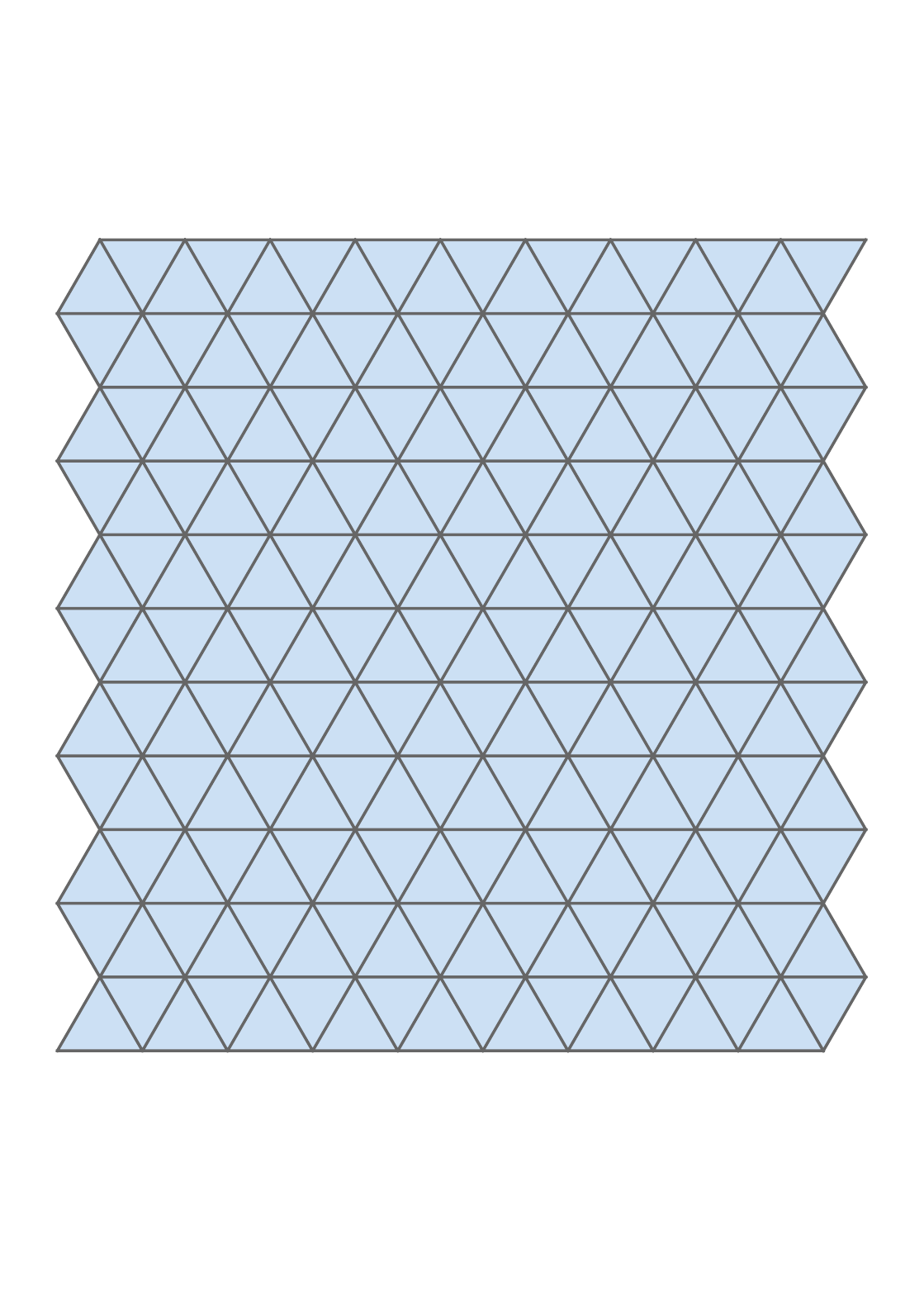} 
\begin{minipage}[b][0.18\textwidth][t]{0.2\textwidth}
 \[\overset{g}{\longrightarrow} \]
\end{minipage}
\includegraphics[width=0.3\textwidth, trim={0 1.7cm 0 1.7cm},clip]{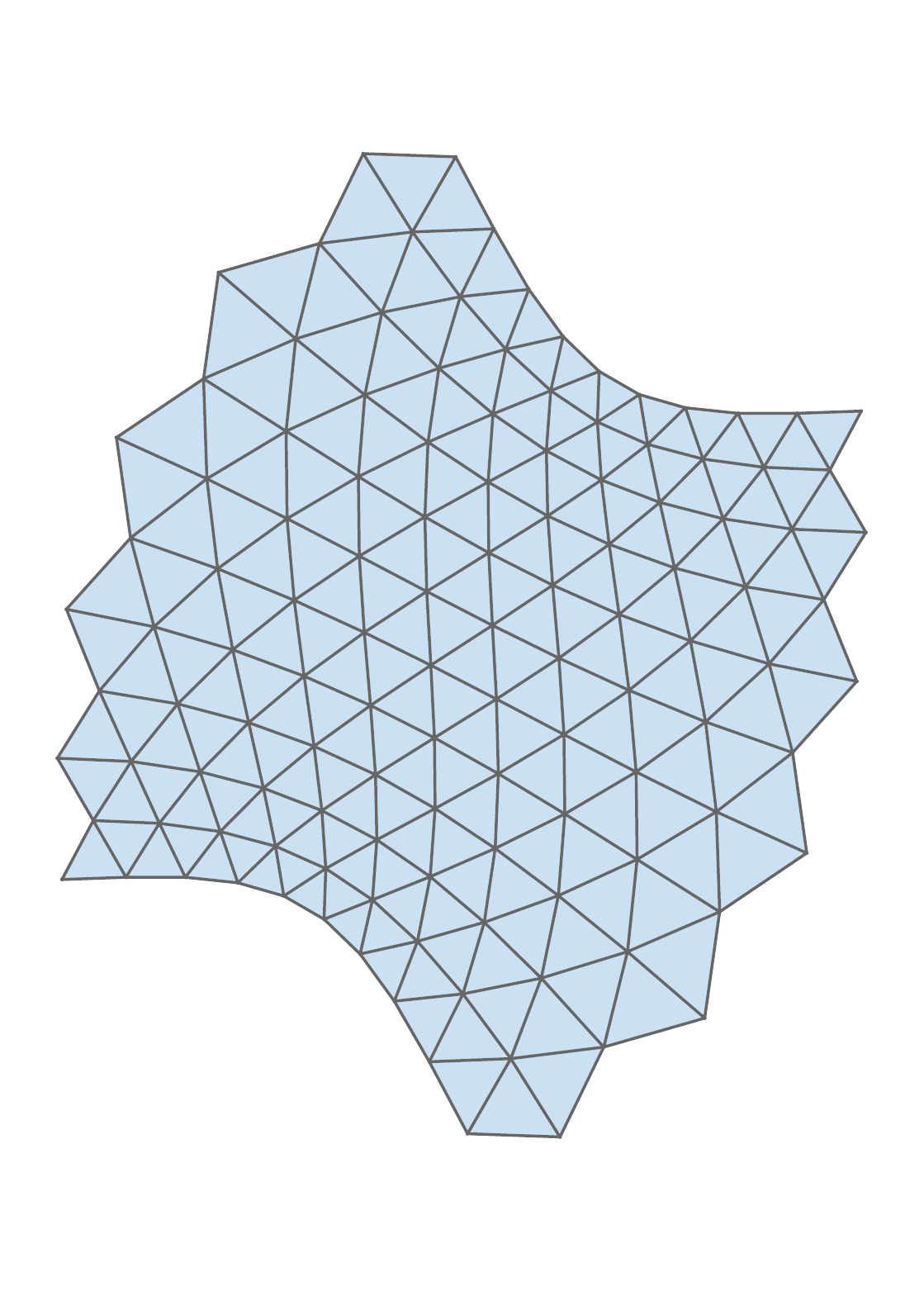}
}
\caption{Example of a discrete conformal PL-map $g$.}\label{FigExMap}
\end{figure}

In fact, the definition of a discrete conformal PL-map relies on the notion of 
discrete conformal triangle meshes. These have been studied by  Luo, Gu, Sun, 
Wu, Guo~\cite{Luo04,Luo13,Luo14}, Bobenko, Pinkall, 
and Springborn~\cite{Boris} and others.

As possible application,
discrete conformal PL-maps can be used for discrete uniformization. The 
simplest 
case is a discrete Riemann mapping theorem, i.e.\ the problem of finding a 
discrete conformal mapping of a simply connected domain onto the unit disc. 
Similarly, we may consider a related Dirichlet problem. Given some function 
$u_\partial$ on the boundary of a subcomplex $T_S$,
find a discrete conformal PL-map whose associated
scale factors agree on the boundary with $u_\partial$. For such a Dirichlet 
problem (with assumptions on $u_\partial$ and $T_S$) we will prove existence as 
part of our convergence theorem.

In this article we present a first answer to the following problem:
{\em Given a smooth conformal map, find a sequence of discrete 
conformal PL-maps which approximate the given map.}
We study this problem on triangular lattices $T$ with acute angles and always 
assume for simplicity that the origin is a vertex.
Denote by $\eps T$ the lattice $T$ scaled by $\eps>0$. 
Using the values of $\log |f'|$, we obtain a discrete conformal PL-map $f^\eps$ 
on a subcomplex of $\eps T$ from a boundary 
value problem for the associated scale factors. More precisely, we prove the 
following approximation result.

\begin{theorem}\label{theoConv}
Let $f:D\to\C$ be a conformal map (i.e.\ holomorphic with $f'\not=0$). Let 
$K\subset D$ be a compact set which is the closure of its simply connected 
interior $int(K)$ and assume that $0\in int(K)$. 
Let $T$ be a triangular lattice with strictly acute angles.
For each $\eps>0$ let $T^\eps_K$ 
be a subcomplex of $\eps T$ whose support is contained in $K$ and is 
homeomorphic to a closed disc. We further assume that $0$ is an 
interior vertex of $T^\eps_K$. Let $e_0=[0,{\mathbb v_{\mathbb 0}}]\in 
E^\eps_K$ be one 
of its incident edges.

Then if $\eps>0$ is small enough (depending on 
$K$, $f$, and $T$) there exists a unique discrete conformal PL-map $f^\eps$ on 
$T^\eps_K$ which satisfies the following two conditions:
\begin{itemize}
 \item The associated scale factors $u^\eps:V^\eps_K\to\R$ satisfy
\begin{equation}\label{eqboundu}
 u^\eps(v)=\log|f'(v)|\qquad \text{for all boundary vertices } v \text{ of } 
V^\eps_K.
\end{equation}
\item The discrete conformal PL-map is normalized according to $f^\eps(0)=f(0)$ 
and $\arg(f^\eps({\mathbb v_{\mathbb 0}})-f^\eps(0))= \arg({\mathbb v_{\mathbb 
0}})+ \arg(f'(\frac{{\mathbb v_{\mathbb 0}}}{2})) 
\pmod{2\pi}$.
\end{itemize}
Furthermore, the following estimates for $u^\eps$ and $f^\eps$ hold for all 
vertices $v\in V^\eps_K$ and points $x$ in the support of $T^\eps_K$ 
respectively with constants 
$C_1,C_2,C_3$ depending only on $K$, $f$, and $T$, but not on $v$ or $x$:
\begin{enumerate}[(i)]
 \item The scale factors $u^\eps$ approximate $\log |f'|$ uniformly with error 
of order~$\eps^2$:
\begin{equation}\label{eqconvu}
 \left|u^\eps(v)-\log|f'(v)|\right|\leq C_1\eps^2.
\end{equation}
\item The discrete conformal PL-mappings $f^\eps$ 
converge to $f$ for $\eps \to 0$ uniformly with error of order~$\eps$:
\begin{equation*}
 \left|f^\eps(x)-f(x)\right|\leq C_2\eps.
\end{equation*}
\item The derivatives of $f^\eps$ (in the interior of the triangles) converge 
to $f'$ uniformly for $\eps \to 0$ with error of order~$\eps$:
\begin{equation*}
 \left|\partial_z f^\eps(x)-f'(x)\right|\leq C_3\eps\qquad 
\text{and} \qquad \left|\partial_{\bar z} f^\eps(x)\right|\leq C_3\eps
\end{equation*}
for all points $x$ in the interior of a triangle $\Delta$ of $T^\eps_K$. Here 
$\partial_z$ and $\partial_{\bar z}$ denote the Wirtinger derivatives applied 
to the linear maps $f^\eps|_\Delta$.
\end{enumerate}
\end{theorem}

Note that the subcomplexes $T^\eps_K$ may be chosen such that they 
approximate the compact set $K$.
Further notice that~\eqref{eqconvu} implies that $u^\eps$ 
converges to
$\log|f'|$ in $C^1$ with error of order $\eps$, in the sense that also
\begin{equation*}
 \left| \frac{u^\eps(v)-u^\eps(w)}{\eps} - \text{Re}\left( 
\frac{f''((v+w)/2)}{f'((v+w)/2)}\right)\right|\leq {\tilde C}\eps
\end{equation*}
 on edges $[v,w]$ uniformly for some constant $\tilde C$.

The proof of Theorem~\ref{theoConv} is given in Section~\ref{secProof}. The 
arguments are based on estimates derived in Section~\ref{secTaylor}. 

The problem of actually computing the 
scale factors $u$ for given boundary values $u_\partial$ such that $u$ gives 
rise to a discrete conformal PL-map (in case it exists) can be 
solved using a variational principle, see~\cite{BorisSch,Boris}. Our 
proof relies on investigations using the corresponding convex functional, see 
Theorem~\ref{theoEu} in Section~\ref{secBasic}.

\begin{remark}
 The convergence result of Theorem~\ref{theoConv} also remains true if linear 
interpolation is replaced with the piecewise projective interpolation schemes 
described in~\cite{Boris, BorisStefanich}, i.e., circumcircle preserving, angle 
bisector preserving and, generally, exponent-t-center preserving for all 
$t\in\R$. 
The proof is the same with only small 
adaptations. This is due to the fact that the image of the vertices is the same 
for all these interpolation schemes and these image points converge 
uniformly to the corresponding image points under $f$ with error of 
order~$\eps$. The estimates for the derivatives similarly follow from 
Theorem~\ref{theoConv}(i).
\end{remark}

\subsection{Other convergence results for discrete conformal maps}

Smooth conformal maps can be characterized in various ways. This leads to 
different notions of discrete conformality. Convergence issues have already
been studied for some of these discrete analogs. We only give a very short 
overview and cite some results of a growing literature.

In particular, linear definitions can be derived 
as discrete versions of the Cauchy-Riemann equations and have a long and 
still developing history. Connections of such discrete mappings to smooth 
conformal functions have been studied for example 
in~\cite{CFL28,LF55,Mer07,ChS12,S13,BS15,W14}.

The idea of characterizing conformal maps as local scale-rotations has lead to
the consideration of circle packings, more precisely to investigations on 
circle packings with the same (given) combinatorics of the 
tangency graph. Thurston~\cite{Thu85} first conjectured the convergence of 
circle packings to the Riemann map, which was then proven 
by~\cite{RS87,HeSch96,HeSch98}.

The theory of circle patterns generalizes the case of circle packings. Also, 
there is a link to integrable structures via isoradial circle 
patterns. The approximation of conformal maps using circle pattens has been 
studied in~\cite{Sch97,DM05,diss,paperconv,LD07}.

The approach taken in this article constructs discrete conformal maps from 
given boundary values. Our approximation 
results and some ideas of the proof are therefore similar to those 
in~\cite{Sch97,diss,paperconv} for circle patterns which also rely on boundary 
value problems.

\section{Some characterizations of associated 
scale factors of discrete conformal PL-maps}\label{secBasic}

Consider a subcomplex $T_S$ of a triangular lattice $T$ and an arbitrary 
function $u:V_S\to\R$. Assign new lengths to the edges according 
to~\eqref{eqdefdiscf} by 
\begin{equation}\label{eqdeftildel}
 \tilde{l}([v,w])=|v-w|\text{e}^{(u(v)+u(w))/2}
\end{equation}
In order to obtain new triangles with these lengths (and ultimately a discrete 
conformal PL-map) the triangle 
inequalities need to hold for the edge lengths $\tilde{l}$ on each triangle.
If we assume this, we can embed the new triangles (respecting orientation) and 
immerse sequences of 
triangles with edge lengths given by $\tilde{l}$ as in~\eqref{eqdeftildel}. In 
order to obtain a discrete conformal PL-map, in particular a local 
homeomorphism, the interior 
angles of the triangles need to sum up to $2\pi$ at each interior vertex. The 
angle at a vertex of a triangle with given side lengths can be calculated. With 
the notation of Figure~\ref{figTriang} we have the half-angle formula
\begin{equation}\label{defalpha1}
 \tan\left(\frac{\alpha}{2}\right) = 
\sqrt{\frac{(-b+a+c)(-c+a+b)}{(b+c-a)(a+b+c)}}
= \sqrt{\frac{1-(\frac{b}{a}-\frac{c}{a})^2}{(\frac{b}{a}+\frac{c}{a})^2-1}}.
\end{equation}
The last expression emphasizes the fact that the angle does not depend on the 
scaling of the triangle. Careful considerations of this angle function
depending on (scaled) side lengths of the triangle form the basis for our 
proof. In particular, we define the function
\begin{equation}\label{defangle}
 \theta(x,y):=2\arctan\sqrt{\frac{1-(\text{e}^{-x/2}
-\text{e}^{-y/2})^2}{(\text{e }^{-x/2} +\text{e}^{-y/2})^2-1}},
\end{equation}
so~\eqref{defalpha1} can be written as
\[\alpha=\theta(x,y)\qquad \text{with}\quad 
\frac{b}{a}=\text{e}^{-x/2}\ \text{ and }\ \frac{c}{a}=\text{e}^{-y/2}.\]

Summing up, we have the following characterization of
scale factors associated to discrete conformal PL-maps.

\begin{proposition}\label{lemTu2}
Let $T_S$ be a subcomplex of a triangular lattice $T$ and $u:V_S\to\R$ a 
function
satisfying the following two conditions.
\begin{enumerate}[(i)]
\item For every triangle $\Delta[v_1,v_2,v_3]$ of $T_S$ the {\em triangle 
inequalities} for $\tilde l$ defined by~\eqref{eqdeftildel} hold, in particular 
\begin{equation}\label{eqtriangineq}
|v_i-v_j|\text{e}^{(u(v_i)+u(v_j))/2}< 
|v_i-v_k|\text{e}^{(u(v_i)+u(v_k))/2} +|v_j-v_k|\text{e}^{(u(v_j)+u(v_k))/2}
\end{equation}
for all permutations $(ijk)$ of $(123)$.
\item For every interior vertex $v_0$ with neighbors
$v_1,v_2,\dots,v_k,v_{k+1}=v_1$ in cyclic order we have
\begin{equation}\label{eqsumalpha}
 \sum_{j=1}^k \theta(\lambda(v_0,v_j,v_{j+1})+
u(v_{j+1})-u(v_0), \lambda(v_0,v_{j+1},v_j) +u(v_j)-u(v_0))=2\pi,
\end{equation}
where $\lambda(v_a,v_b,v_c)= 2\log(|v_b-v_c|/|v_a-v_b|)$ for a triangle
$\Delta[v_a,v_b,v_c]$.
\end{enumerate}
Then there is a discrete conformal PL-map (unique up to post-composition with 
Euclidean motions) such that its associated scale factors are the given 
function $u:V_S\to\R$.

Conversely, given a discrete conformal PL-map on a subcomplex $T_S$ of a 
triangular lattice $T$, its associated scale factors $u:V_S\to\R$ satisfy 
conditions~(i) and~(ii).
\end{proposition}

In order to obtain discrete conformal PL-maps from a given smooth conformal 
map we will consider a Dirichlet problem for the associated scale factors. 
Therefore we will apply a theorem from~\cite{Boris} which characterizes the 
scale factors $u$ for given boundary values using a variational principle for a 
functional $E$ defined in~\cite[Section~4]{Boris}. Note that we will not need 
the exact expression for $E$ but only the formula for its partial derivatives. 
In fact, the vanishing of these derivatives is equivalent to the necessary 
condition~\eqref{eqsumalpha} for the scale factors to correspond to a 
discrete conformal PL-map.

\begin{theorem}[\cite{Boris}]\label{theoEu}
Let $T_S$ be a subcomplex of a triangular lattice and let 
$u_\partial:V_\partial\to\R$ be a 
function on the boundary vertices $V_\partial$ of $T_S$. Then the solution 
$\tilde u$ (if it exists)
of equation~\eqref{eqsumalpha} at all interior vertices with ${\tilde 
u}|_{V_\partial}=u_\partial$ is the 
unique argmin of a locally strictly convex functional $E(u)=E_{T_S}(u)$
which is defined for functions $u:V\to\R$ satisfying the
inequalities~\eqref{eqtriangineq}. 

The partial derivative of $E$ with respect
to $u_i=u(v_i)$ at an interior vertex $v_i\in V_{int}$ with $k$ neighbors
$v_{i_1},v_{i_2},\dots,v_{i_k}v_{i_{k+1}}=v_{i_1}$ in cyclic order is
\begin{equation}
 \frac{\partial E}{\partial u_i}(u) = 2\pi -  \sum_{j=1}^k
\theta(2\log\left(\frac{l_{i_{j+1},i_j}}{l_{i,i_{j+1}}}\right)+ u_{i_j}- u_i,
2\log\left(\frac{l_{i_{j+1},i_j}}{l_{i,i_{j}}}\right)+ u_{i_{j+1}} -u_i),
\end{equation}
where $l_{j,k}=|v_j-v_k|$.

By Proposition~\ref{lemTu2} such a solution $\tilde u$ are then scale factors 
associated to a discrete conformal PL-map.
\end{theorem}

\begin{remark}
The functional $E$ can be extended to a convex continuously
differentiable function on $\R^V$, see~\cite{Boris} for details.
\end{remark}

\section{Taylor expansions}\label{secTaylor}
We now examine the effect when we take $u=\log|f'|$ as `scale 
factors', i.e.\ 
for each triangle we multiply the length $|v-w|$ of an edge $[v,w]$ by the 
geometric mean $\sqrt{|f'(v)f'(w)|}$ of $|f'|$ at the vertices.
The proof of Theorem~\ref{theoConv} is based on the idea that $u=\log|f'|$ 
almost satisfies the
conditions for being the associated scale factors of an discrete conformal 
PL-map, that is conditions~(i) and~(ii) of Proposition~\ref{lemTu2}, and 
therefore is close to the exact solution $u^\eps$. 

To be precise, suppose that $\eps T$ is the equilateral 
triangulation of the plane. Assume without loss of generality that the
edge lengths equal $\frac{\sqrt{3}}{2}\eps>0$ and edges are parallel to 
$\text{e}^{ij\pi/3}$ for $j=0,1,\dots, 5$. 
Let the conformal function $f$, the compact set $K$, and the subcomplexes 
$T^\eps_K$ (with vertices $V^\eps_K$ and edges $E^\eps_K$) be given as in 
Theorem~\ref{theoConv}.
Let $v_0\in V^\eps_{K, \text{int}}$ be an interior vertex. Here and below 
$V^\eps_{K,\text{int}}$ denotes the set of interior vertices having six 
neighbors in $V^\eps_K$. Denote the neighbors of $v_0$ by $v_j= v_0+\eps 
\frac{\sqrt{3}\text{e}^{ij\frac{\pi}{3}}}{2}$ and consider the triangle 
$\Delta_j= \Delta[v_0,v_j,v_{j+1}]$ for some $j\in \{0,1,\dots, 
5\}$. Taking $u=\log|f'|$, we obtain edge lengths of a new triangle ${\tilde 
\Delta}_j$, i.e.\ satisfying~\eqref{eqtriangineq}, if $\eps$ is small enough. 
Then 
the angle in ${\tilde \Delta}_j$ at the image vertex of $v_0$ is given by
\[ \theta(\log|f'(v_0+\eps \textstyle
\frac{\sqrt{3}\text{e}^{ij\frac{\pi}{3}}}{2})| -\log|f'(v_0)|,\, 
\log|f'(v_0+\eps 
\frac{\sqrt{3}\text{e}^{i(j+1)\frac{\pi}{3}}}{2})|
-\log|f'(v_0)|) \]
according to~\eqref{defangle}. Summing up these angles --- that is inserting
$\log|f'|$ into~\eqref{eqsumalpha} instead of $u$ 
at an interior vertex $v_0\in V^\eps_{K, \text{int}}$ --- we obtain the 
function
\begin{multline*}
 {\cal S}_{v_0}(\eps)= \\
\sum_{j=0}^5 \theta(\log|f'(v_0+\eps \textstyle
\frac{\sqrt{3}\text{e}^{ij\frac{\pi}{3}}}{2})| -\log|f'(v_0)|,\, 
\log|f'(v_0+\eps 
\frac{\sqrt{3}\text{e}^{i(j+1)\frac{\pi}{3}}}{2})| -\log|f'(v_0)|)
\end{multline*}
We are interested in the Taylor expansion of ${\cal S}_{v_0}$ in $\eps$.
The symmetry of the lattice $T$ implies that ${\cal S}_{v_0}$ is 
an even function, so the expansion contains only even powers of $\eps^n$.
Using a computer algebra program  we arrive at
\begin{align}\label{eqTaylorexp}
 {\cal S}_{v_0}(\eps)= 2\pi + C_{v_0}\eps^4 +\Or(\eps^6).
\end{align}
Here and below, the notation $h(\eps)=\Or(\eps^n)$ means that there is a
constant $\cal C$,
such that $|h(\eps)|\leq {\cal C}\eps^n$ holds for all small enough $\eps>0$.
The constant of the $\eps^4$-term is
\[C_{v_0}= \textstyle
-\frac{3\sqrt{3}}{32} \text{Re}\left(S(f)(v_0) 
\overline{\left({\textstyle \frac{f''}{f'}}\right)'}(v_0)\right),\] 
where 
$S(f)=\left(\frac{f''}{f'}\right)' -\frac{1}{2} 
\left( \frac{f''}{f'}\right)^2$ is the Schwarzian derivative of $f$.
We will not need the exact form of this constant, but only the fact 
that it is bounded on $K$.

Analogous results to~\eqref{eqTaylorexp} hold for all triangular
lattices $\eps T$ with edge lengths $a^\eps=\eps\sin\alpha$, 
$b^\eps=\eps\sin\beta$, $c^\eps=\eps\sin\gamma$, also if the angles
are larger than $\pi/2$. We assume without loss of generality the edge 
directions being parallel to $1$, 
$\text{e}^{i\alpha}$ and $\text{e}^{i(\alpha+\beta)}$. 
Arguing as above, we consider the function
\begin{align*}
 {\cal S}_{v_0}(\eps) =&\quad \theta(\textstyle 
2\log\frac{\sin\alpha}{\sin\gamma} +\log|\frac{f'(v_0+\eps 
\sin\beta)}{f'(v_0)}|,\, 2\log\frac{\sin\alpha}{\sin\beta}+
\log|\frac{f'(v_0+\eps \sin\gamma  \text{e}^{i \alpha})}{f'(v_0)}|) \\
&+ \theta(\textstyle 2\log\frac{\sin\beta}{\sin\alpha}+ \log|\frac{f'(v_0+\eps 
\sin\gamma\,  \text{e}^{i \alpha})}{f'(v_0)}|, 
\, 2\log\frac{\sin\beta}{\sin\gamma}+ \log|\frac{f'(v_0+\eps \sin\alpha\,  
\text{e}^{i (\alpha+\beta)})}{f'(v_0)}|) \\
&+ \theta(\textstyle 2\log\frac{\sin\gamma}{\sin\beta}+ \log|\frac{f'(v_0+\eps 
\sin\alpha\,  \text{e}^{i (\alpha+\beta)})}{f'(v_0)}|,\, 
2\log\frac{\sin\gamma}{\sin\alpha}+\log|\frac{f'(v_0-\eps 
\sin\beta)}{f'(v_0)}|) \\
&+ \theta(\textstyle 2\log\frac{\sin\alpha}{\sin\gamma}+ \log|\frac{f'(v_0-\eps 
\sin\beta)}{f'(v_0)}|, \, 2\log\frac{\sin\alpha}{\sin\beta}+
\log|\frac{f'(v_0-\eps \sin\gamma \, \text{e}^{i \alpha})}{f'(v_0)|}) \\
&+ \theta(\textstyle 2\log\frac{\sin\beta}{\sin\alpha}+ \log|\frac{f'(v_0-\eps 
\sin\gamma\,  \text{e}^{i \alpha})}{f'(v_0)}|,\, 
2\log\frac{\sin\beta}{\sin\gamma}+ \log|\frac{f'(v_0-\eps \sin\alpha\,  
\text{e}^{i (\alpha+\beta)})}{f'(v_0)}|) \\
&+ \theta(\textstyle 2\log\frac{\sin\gamma}{\sin\beta}+ \log|\frac{f'(v_0-\eps 
\sin\alpha \, \text{e}^{i (\alpha+\beta)})}{f'(v_0)}|,\, 
2\log\frac{\sin\gamma}{\sin\alpha}+ \log|\frac{f'(v_0+\eps 
\sin\beta)}{f'(v_0)}|).
\end{align*}
Again, ${\cal S}_{v_0}$ is an even function.
Using a computer algebra program  we arrive at
\begin{align}
 {\cal S}_{v_0}(\eps)= 2\pi + C_{v_0}\eps^4 +\Or(\eps^6),
\end{align}
with corresponding constant
\begin{multline*}
\textstyle
C_{v_0}= -\frac{\sin\alpha \sin\beta \sin\gamma}{4}\; \text{Re}\left(
S(f)(v_0) \overline{\left(\frac{f''}{f'}\right)'}(v_0)\right. \\
\textstyle
\left. +c(\alpha,\beta,\gamma)
\left(\frac{1}{2}\left(\frac{f''}{f'}\right)^2 \left(\frac{f''}{f'}\right)' 
-\frac{1}{3}\left(\frac{f''}{f'}\right)''' \right)\right),
\end{multline*} 
where $c(\alpha,\beta,\gamma)=\cos\beta\sin^3\beta 
+\cos\gamma\sin^3\gamma\text{e}^{2i\alpha} 
+\cos\alpha\sin^3\alpha\text{e}^{2i(\alpha+\beta)}$.

Our key observation is that we can control the sign of the 
$\Or(\eps^4)$-term in~\eqref{eqTaylorexp} if we
replace $\log|f'(x)|$ by $\log|f'(x)|+a\eps^2|x|^2$,
where $a\in\R$ is some suitable constant. In particular, for positive constants 
$M^\pm,C^\pm$ consider the functions
\begin{align*}
w^\pm &= \log|f'| +q^\pm
&\text{ with } q^\pm(v) &=\begin{cases}\pm\eps^2(M^\pm -C^\pm|v|^2) & \text{for
}v\in
V^\eps_{K,\text{int}}, \\
        0 & \text{for }v\in \partial V^\eps_K.
       \end{cases}
\end{align*}
Here and below $\partial V^\eps_K$ denotes the set of boundary vertices 
of $V^\eps_K$.

Then we obtain for equilateral triangulations with edge length 
$\frac{\sqrt{3}}{2}\eps$ the following Taylor expansion
for all interior vertices $v_0\in V^\eps_{K,\text{int}}$ whose neighbors 
are also in $V^\eps_{K,\text{int}}$:
\begin{multline}
\sum_{j=0}^5 \theta(w^\pm(v_0+\eps \textstyle \frac{\sqrt{3}}{2}
\text{e}^{ij\frac{\pi}{3}}) -w^\pm(v_0), w^\pm(v_0+\eps \frac{\sqrt{3}}{2}
\text{e}^{i(j+1)\frac{\pi}{3}}) -w^\pm(v_0)) \\
\textstyle
= 2\pi + (C_{v_0}\mp \frac{3\sqrt{3}}{2}C^\pm)\eps^4 
+\Or(\eps^5).\label{eqestw}
\end{multline}
Again, analogous results hold for all regular triangular lattices, where the
corresponding $\Or(\eps^4)$-term then is 
\[C_{v_0}\mp 4\textstyle \sin\alpha \sin\beta \sin\gamma\; C^\pm.\]

For interior vertices $v_0\in V^\eps_{K,\text{int}}$ which are incident 
to $k$ 
boundary vertices we obtain instead of the right-hand side of~\eqref{eqestw}:
\[2\pi \mp 
k\textstyle\frac{\sqrt{3}}{4} (M^\pm-C^\pm|v_0|^2)\eps^2 +\Or(\eps^4).\]
For general triangular lattices we get for every edge $e=[v_0,v_j]$ 
which is incident to a boundary vertex $v_j\in\partial V^\eps_K$ a term 
$\mp(M^\pm-C^\pm|v_0|^2)\cos\varphi_e\sin\varphi_e\,\eps^2$ where $\varphi_e$ 
is 
the angle opposite to the edge $e$, see Figure~\ref{FigDefDual}.

The following lemma summarizes the main properties of $w^\pm$ which follow from 
the definition of $w^\pm$ together with the preceding estimates.

\begin{lemma}\label{lemwpm}
 $w^\pm$ satisfies the boundary condition $w^\pm|_{\partial V^\eps_K} =
\log|f'| \big|_{\partial V^\eps_K}$.

Furthermore, $C^\pm>0$ and $M^\pm>0$ can be chosen such that for all $\eps$ 
small enough and all $v_0\in V^\eps_{K, \text{int}}$:
\begin{enumerate}[(i)]
\item $q^+(v_0)>0$ and $q^-(v_0)<0$\smallskip
 \item If $v_1, v_2,\dots, v_6,v_7=v_1$ denote the chain of neighboring
vertices of $v_0$ in cyclic order and $\lambda(v_a,v_b,v_c)=
2\log(|v_b-v_c|/|v_a-v_b|)$ for any triangle $\Delta[v_a,v_b,v_c]$, we 
have
\begin{align*}
\sum_{j=1}^6 \theta(&\lambda(v_0,v_{j+1},v_j)
+w^+(v_{j})- w^+(v_0), \lambda(v_0,v_j,v_{j+1}) +w^+(v_{j+1})-w^+(v_0)) < 
2\pi,\\[1ex]
\sum_{j=1}^6 \theta(&\lambda(v_0,v_{j+1},v_j)
+w^-(v_{j})- w^-(v_0),\lambda(v_0,v_j,v_{j+1}) +w^-(v_{j+1})-w^-(v_0)) > 2\pi
\end{align*}
\end{enumerate}
The choices of $C^\pm$ and $M^\pm$ only depend on $f$ (and its derivatives), 
$K$, and on the angles of the triangular lattice $T$.
\end{lemma}

In analogy to the continuous case we interpret equation~\eqref{eqsumalpha} as a 
non-linear Laplace equation for $u$. In this spirit $w^+$ may be taken as 
superharmonic function and $w^-$ as subharmonic function.

\section{Existence of discrete conformal PL-maps and 
estimates}\label{secProof}

The functions $w^\pm$ have been introduced in order to ``catch'' the
solution $u^\eps$ in the following compact set:
\begin{multline*}
 W^\eps=\{u:V^\eps_K\to\R\ |\ u(v) =
\log|f'(v)| \text{ for all } v\in\partial V^\eps_K,\\ w^-(v)\leq 
u(v)\leq w^+(v) \text{
for all } v\in V^\eps_{K, \text{int}} \}.
\end{multline*}
Note that $W^\eps$ is a $n$-dimensional interval in $\R^n$ for $n=|V^\eps_K|=$ 
number of vertices, if we identify a
function $u:V^\eps_K\to\R$ with the vector of its values $u(v_i)$.
Also, for neighboring vertices $v_i\sim v_j$ and $u\in W^\eps$ we have
$u(v_j)-u(v_i)=\Or(\eps)$.
Therefore, $u\in W^\eps$ satisfies the triangle
inequalities~\eqref{eqtriangineq} if $\eps$ is small enough.

Our aim is to show that for $\eps$ small enough there exists a function 
$u^\eps$ satisfying conditions~(i) and~(ii) of Proposition~\ref{lemTu2} and 
$u^\eps(v)=\log|f'(v)|$ for all boundary vertices $v\in\partial V^\eps_K$. This 
function then defines a discrete conformal PL-map $f^\eps$ 
(uniquely if we use the normalization of Theorem~\ref{theoConv}). 

\begin{theorem}\label{theoexist}
Assume that all angles of the triangular lattice $T$ are strictly 
smaller than $\pi/2$.
 There is an $\eps_0>0$
(depending on $f$, $K$ and the triangulation parameters)
such that for all $0<\eps<\eps_0$ the minimum of the functional $E$
(see Theorem~\ref{theoEu}) with boundary conditions~\eqref{eqboundu} is
attained in $W^\eps$.
\end{theorem}

\begin{corollary}
 For all $0<\eps<\eps_0$ there exists a discrete conformal PL-map on
$T^\eps_K$ whose associated scale factors satisfy the boundary 
conditions~\eqref{eqboundu}.
\end{corollary}

The proof of Theorem~\ref{theoexist} follows from Lemma~\ref{lemW} below.  
It is based on Theorem~\ref{theoEu} and on monotonicity estimates of the 
angle function $\theta(x,y)$ defined in~\eqref{defangle}. It is only here where 
we need the assumption that all angles of
the triangular lattice $T$ are strictly smaller than $\pi/2$.

\begin{lemma}[Monotonicity lemma]\label{lemMonalpha}
 Consider the star of a vertex $v_0$ of a triangular lattice 
$T$ and its neighboring vertices $v_1,\dots, v_6,v_7=v_1$ in cyclic order.
Denote $\lambda_{0,k}:=2\log(|v_{k+1}-v_{k}|/|v_{0}-v_{k}|)$.
Assume that all triangles $\Delta(v_0,v_k,v_{k+1})$ are strictly acute angled, 
i.e.\ all angles $<\pi/2$.

Then there exists $\eta_0>0$, depending on the $\lambda$s, such that for all
$0\leq \eta_1,\dots,\eta_6$, $\eta_7=\eta_1<\eta_0$ there holds
\[\sum_{k=1}^6 \theta(\lambda_{0,k}+\eta_{k},
\lambda_{0,k+1}+\eta_{k+1})\geq \sum_{k=1}^6 \theta(\lambda_{0,k},
\lambda_{0,k+1}),\]
and for all $0\geq \eta_1,\dots,\eta_6,\eta_7=\eta_1>-\eta_0$ we have
\[\sum_{k=1}^6 \theta(\lambda_{0,k}+\eta_{k},
\lambda_{0,k+1}+\eta_{k+1})\leq \sum_{k=1}^6 \theta(\lambda_{0,k},
\lambda_{0,k+1}).\]
\end{lemma}
\begin{proof}
First, consider a single acute angled triangle.
 Observe that with the notation of Figure~\ref{figTriang}:
\[ \frac{\partial \beta}{\partial a} = -\frac{1}{a}\cot\gamma.\]
Thus, we easily deduce that
 \[\left. \frac{\partial}{\partial \eps}
\theta(2\log(\frac{a}{c})+\eps,2\log(\frac{a}{b}))\right|_{\eps=0}= \frac{1}{2} 
\cot\gamma.\]
Now the claim follows by Taylor expansion.
\hfill{} \qed
\end{proof}

\begin{lemma}\label{lemW}
 There is an $\eps_0>0$
such that for all $0<\eps<\eps_0$ the negative gradient $-\text{grad}(E)$ on the
boundary of $W^\eps$ points into the interior of $W^\eps$.
\end{lemma}
\begin{proof}
 For notational simplicity, set $u_k=u(v_k)$, $w_k^\pm=w^\pm(v_k)$ for
vertices $v_k\in V^\eps_K$ and 
$\lambda_{a,b,c}=2\log(|v_b-v_c|/|v_a-v_b|)$.

Consider $\text{grad} (E)$ on a boundary face $W_i^+=\{u\in W^\eps :
u_i=w_i^+\}$ of the $n$-dimensional interval $W^\eps$. Let $v_1,\dots, 
v_6,v_7=v_1$ denote the neighbors of $v_i$ in cyclic order.
Note that $w_j^+-w_j^-=\eps^2(M^++M^--(C^++C^-)|v_j|^2)$ for all vertices $v_j$.
As $K$ is compact we may assume that $0<\eps_0$ is such that $w_j^+-w_j^-\leq
\eps$ for $0<\eps<\eps_0$. 
Then using the properties of $w^+$ and $u$ we obtain from
Lemma~\ref{lemMonalpha} and Lemma~\ref{lemwpm}
\begin{align*}
 \frac{\partial E}{\partial u_i}(u) &= 2\pi - \sum_{j=0}^5
\theta(\lambda_{i,{j+1},j} +\underbrace{u_j- \underbrace{u_i}_{=w_i^+}}_{\leq
w_j^+-w_i^+} ,\lambda_{i,j,{j+1}} +
\underbrace{u_{j+1} -\underbrace{u_i}_{=w_i^+}}_{\leq
w_{j+1}^+-w_i^+})\\
&\geq 2\pi - \sum_{j=0}^5 \theta(\lambda_{i,{j+1},j} +w_j^+-w_i^+,
\lambda_{i,j,{j+1}} +w_{j+1}^+-w_i^+) \\
&> 0.
\end{align*}
An analogous estimate holds for boundary faces $W_i^-$.
\hfill{} \qed
\end{proof}

We are now ready to deduce our convergence theorem.

\begin{proof}[Proof of Theorem~\ref{theoConv}]
The existence part follows from Theorem~\ref{theoexist}. The uniqueness is 
obvious as the translational and rotational freedom of the image of $f^\eps$ is 
fixed using values of $f$.

We now deduce the remaining estimates.

\smallskip
\noindent
{\it Part~(i):}
Together with the definition of $w^\pm$, Theorem~\ref{theoexist} implies 
that for $\eps>0$ small enough and all vertices $v\in V^\eps_K$
\begin{multline*}
-\eps^2(M^--C^-|v|^2) \\
\leq w^-(v) -\log|f'(v)| \leq u^\eps(v)-\log|f'(v)|\leq w^+(v) -\log|f'(v)|\\
\leq \eps^2(M^+-C^+|v|^2).
\end{multline*}
As $K$ is compact, this implies estimate~\eqref{eqconvu}.
\smallskip

\noindent
{\it Part~(ii):}
Given the scale factors $u^\eps$ associated to the discrete conformal PL-map 
$f^\eps$ on $T^\eps_K$, we can in every image triangle determine the 
interior angles (using for example~\eqref{defalpha1}). In particular, we begin 
by deducing from estimate~\eqref{eqconvu} the change of these 
interior angles of the triangles.

Recall that for acute angled triangles the center of the circumcircle lies in 
the interior of the triangle. Joining these centers for incident triangles 
leads to an embedded regular graph $\eps T^*=(\eps V^*,\eps E^*)$ which 
is dual to the given triangular lattice $\eps T$. In particular, the vertices 
$\eps V^*$ are identified with the 
centers of the circumcircles of the triangles of $\eps T$.
Furthermore, each edge $e^*\in (\eps E^*)$ intersects exactly one edge $e\in 
(\eps E)$ orthogonally, so $e$ and $e^*$ are dual, see Figure~\ref{FigDefDual}. 
\begin{figure}[t]
\begin{center}
 \input{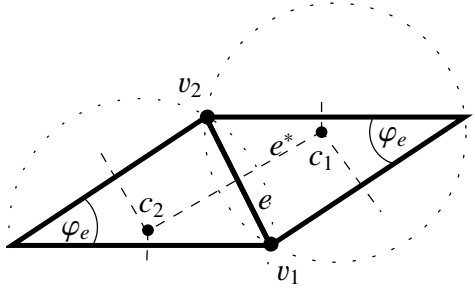}
\end{center}
\caption{Two adjacent triangles of the triangular lattice $\eps T$ and 
orthogonal edges $e\in (\eps E)$ (solid) and $e^*\in (\eps E^*)$ 
(dashed).}\label{FigDefDual}
\end{figure}
Consider an edge $e=[v_1,v_2]\in E^\eps_K$ with dual edge 
$e^*=[c_1,c_2]$. Their lengths are related by 
$|c_2-c_1|= |v_2-v_1|\cot\varphi_e$, where $\varphi_e$ denotes the angle 
opposite to $e$ in $\eps T$. Furthermore we obtain 
\begin{align}
\cot\varphi_e(\log|f'(v_2)|-\log|f'(v_1)|) 
&=\cot\varphi_e\text{Re}((\log f')'(v_1)(v_2-v_1)) +\Or(\eps^2) \notag\\
&= \cot\varphi_e\text{Im}((\log f')'(v_1)i(v_2-v_1)) +\Or(\eps^2)\notag\\
&=\text{Im}((\log f')'(v_1)(c_2-c_1)) +\Or(\eps^2) \notag\\
&=2\text{Im}((\log f')'(v_1)(c_2-v_1)) \notag\\
&\quad +2\text{Im}((\log f')'(v_1)(v_1-\textstyle\frac{c_2+c_1}{2})) 
+\Or(\eps^2) \notag\\
&=\textstyle 2\arg f'(c_2)-2\arg 
f'(\underbrace{\textstyle\frac{c_2+c_1}{2}}_{=\frac{v_2+v_1}{2}}) 
+\Or(\eps^2)
\label{eqarg1} \\
&=2\arg f'(\textstyle\frac{v_2+v_1}{2}) -\textstyle 2\arg f'(c_1)+\Or(\eps^2),
\label{eqarg2}
\end{align}
where we have chosen the notation such that $(v_2-v_1)i=(c_2-c_1)\tan\varphi_e$.

Now we estimate the change of the angles in a triangle of $T^\eps_K$ compared 
with its image triangle under $f^\eps$.
Assume given a triangle $\Delta[v_0,v_1,v_2]$ and denote $e_1=[v_0,v_1]$ and 
$e_2=[v_0,v_2]$. Denote the angle at $v_0$ by 
$\theta_0=\theta(\lambda_1,\lambda_2)$, where $l_{e_{j}}=|v_j-v_0|$ and
$\lambda_j=2\log(|v_1-v_2|/l_{e_{j+1}})$ for $j=1,2$ and 
$e_3=e_1$. 
Consider the Taylor expansion
\[\theta(\lambda_1 +x_1\eps,\lambda_2 +x_2\eps)= \theta_0 +\eps( 
\frac{\cot\varphi_{e_1}}{2}x_1 +\frac{\cot\varphi_{e_2}}{2}x_2) + \Or(\eps^2).\]
We apply this estimate for the bounded terms 
\[x_j= \frac{u^\eps(v_j)-u^\eps(v_0)}{\eps} = 
\frac{\log|f'(v_j)|-\log|f'(v_0)|}{\eps} +\Or(\eps)\] 
for $j=1,2$. 
Denote by $\delta +\theta_0\in(0,\pi)$ the angle at the image point of $v_0$ in 
the image triangle $f^\eps(\Delta[v_0,v_1,v_2])$.
Then by~\eqref{eqarg1} and~\eqref{eqarg2}
the change of angle $\delta$ is given by
\begin{equation}\label{eqdeltalpha}
 \delta = \textstyle\arg f'(\frac{v_2+v_0}{2})-\arg f'(\frac{v_0+v_1}{2}) 
+ \Or(\eps^2)
\end{equation}
This local change of angles is related to the angle $\psi^\eps(e)$ by which 
each edge $e$ of $T^\eps_K$ has to be rotated to obtain the corresponding 
image edge $f^\eps(e)$ (or, more precisely, a parallel edge). The function 
$\psi^\eps$ may be 
defined globally on $E^\eps_K$ such that in the above notation the 
change of the angle at $v_0$ is given as 
$\delta = \psi^\eps(e_2)-\psi^\eps(e_1) \in(-\pi,\pi)$.
We fix the value of $\psi^\eps$, that is the rotational freedom of 
the image of $T^\eps_K$ under $f^\eps$ at the edge $e_0$ according to 
$\arg f'$, see Theorem~\ref{theoConv}. Then we take shortest simple paths and 
deduce from~\eqref{eqdeltalpha} that each edge $e=[v_j,v_{j+1}]\in E^\eps_K$ 
is rotated counterclockwise by 
\[\psi^\eps(e)= 
\arg f'(\frac{v_j+v_{j+1}}{2}) + \Or(\eps).\] 
This implies together 
with~\eqref{eqconvu} that for all edges $e=[v_j,v_{j+1}]\in E^\eps_K$ we 
have uniformly
\begin{equation}\label{eqapprf'}
\textstyle \log f'(\frac{v_j+v_{j+1}}{2}) - 
\frac{u^\eps(v_j)+u^\eps(v_{j+1})}{2} 
-i\psi^\eps(e)=\Or(\eps).
\end{equation}
Therefore the difference of the smooth and discrete 
conformal maps at vertices $v_0\in V^\eps_K$ satisfies uniformly
\[f(v_0)-f^\eps(v_0)= O(\eps)\]
by suitable integration along shortest simple paths from the reference 
point as above. This 
estimate then also holds for all points in the support of $T^\eps_K$ and 
$\eps\to 0$.
\smallskip

\noindent
{\it Part~(iii):}
As last step we consider the derivatives of $f^\eps$ restricted to a 
triangle. 

Assume given a triangle $\Delta[v_0,v_1,v_2]$ in $T^\eps_K$. As $f^\eps$ is 
piecewise linear its restriction to $\Delta=\Delta[v_0,v_1,v_2]$ is the 
restriction of an $\R$-linear map $L_\Delta$. This map can be written for 
$z\in\C$ as 
\[L_\Delta(z)= f^\eps(v_0) + a\cdot (z-v_0) +b\cdot \overline{(z-v_0)},\] 
where the constants $a,b\in\C$ are determined from the 
conditions $L_\Delta(v_j)= f^\eps(v_j)$ for $j=0,1,2$. Straightforward 
calculation gives
\begin{align*}
 \partial_z L_\Delta &= a= \frac{(f^\eps(v_2)-f^\eps(v_0)) \overline{(v_1-v_0)} 
-(f^\eps(v_1)-f^\eps(v_0)) \overline{(v_2-v_0)}}{\overline{(v_1-v_0)}(v_2-v_0) 
-(v_1-v_0)\overline{(v_2-v_0)}} \\
\partial_{\bar z} L_\Delta &= b= \frac{(f^\eps(v_2)-f^\eps(v_0))(v_1-v_0) 
-(f^\eps(v_1)-f^\eps(v_0))(v_2-v_0)}{\overline{(v_1-v_0)}(v_2-v_0) 
-(v_1-v_0)\overline{(v_2-v_0)}}.
\end{align*}
Note that by definition of $f^\eps$ and $\psi^\eps$ we know that
\[f^\eps(v_j)-f^\eps(v_0)= (v_j-v_0)\text{e}^{(u^\eps(v_j)+u^\eps(v_0))/2 
+i\psi^\eps([v_j,v_0])},\] 
where we use the rotation function $\psi^\eps$ on the 
edges as defined in the previous part~(ii) of the proof. Now~\eqref{eqapprf'} 
together with the above expressions of $a$ and $b$ immediately implies the 
desired estimates
\[ \partial_z f^\eps|_\Delta(z)=\partial_z L_\Delta(z) =f'(z)+\Or(\eps)
 \quad \text{and}\quad
\partial_{\bar z} f^\eps|_\Delta(z)= \partial_{\bar z} L_\Delta(z)= \Or(\eps).
\]
uniformly on the triangle $\Delta=\Delta[v_0,v_1,v_2]$. Also, the constants in 
the estimate do not depend on the choice of the triangle.
 This finishes the proof.
\hfill{} \qed
\end{proof}

\begin{remark}
Theorem~\ref{theoConv} focuses on a particular way to 
approximate a given conformal map $f$ by a sequence of discrete conformal 
PL-maps. 
Namely, we consider corresponding smooth and discrete Dirichlet boundary value 
problems and compare the solutions.
 There is of course a corresponding problem for Neumann boundary conditions, 
i.e.\ prescribing angle sums of the triangles at boundary vertices using $\arg 
f'$. Also, there is a corresponding 
variational description for conformally equivalent triangle meshes or discrete 
conformal PL-maps in terms of angles, see~\cite{Boris}. 
But unfortunately, the presented methods for a convergence proof seem not to 
generalize in a straightforward manner to this case, as the order of the 
corresponding Taylor expansion is lower.
\end{remark}
\index{Convergence!triangular lattice|)}

\section*{Acknowledgement}
This research was supported by the DFG Collaborative Research Center TRR 109, 
``Discretization in Geometry and Dynamics''. 

\bibliographystyle{spmpsci}      
\bibliography{paperconv}

\end{document}